\documentclass[12pt]{amsart}
\usepackage{latexsym, enumerate}
\usepackage{amssymb,amsmath,mathtools}

\usepackage{color}

\usepackage{amsthm}

\usepackage{todonotes}

\usepackage{multirow}

\mathtoolsset{showonlyrefs}

\newcommand{\rd}{{\rm d}}
\newcommand{\e}{{\rm e}}
\newcommand{\Ran}{\mathop{\rm Ran}}

\newcommand{\R}{{\mathbb R}}

\newcommand\beq{\begin{equation}}
\newcommand\eeq{\end{equation}}
\newcommand{\beqnt}{\begin{equation*}}
\newcommand{\eeqnt}{\end{equation*}}

\newcommand\re{\mathrm{Re}\,}
\newcommand\im{\mathrm{Im}\,}
\newcommand\I{\mathrm{i}}

\DeclareMathOperator{\tr}{Tr}

\newtheorem{theorem}{Theorem}[section]

\newtheorem{lemma}[theorem]{Lemma}

\begin{document}

\title[Resolvent estimates in Schatten spaces]{Resolvent estimates in Schatten spaces for Laplace-Beltrami operators on compact manifolds}

 \author{Jean-Claude Cuenin}
 \address{Department of Mathematical Sciences, Loughborough University, Loughborough,
 Leicestershire, LE11 3TU United Kingdom}
 \email{J.Cuenin@lboro.ac.uk}
 
 
\begin{abstract}
We prove resolvent estimates in Schatten spaces for Laplace-Beltrami operators on compact manifolds at the critical exponent. Our proof only uses known bounds for the Hadamard parametrix. 
\end{abstract}

\maketitle

\section{Introduction and main result}
Let $(M,g)$ be a compact boundaryless Riemannian manifold of dimension $n\geq 2$, with smooth metric $g$, and consider the negative Laplace--Beltrami operator $\Delta_g$ on $M$. The aim of this note is to prove a Schatten space version of the following resolvent estimate,
\begin{align}\label{Frank and Schimmer}
\|(-\Delta_g-z)^{-1}\|_{L^{q'}(M)\to L^{q}(M)}\leq C_{\delta} |z|^{\sigma(q)-\frac{1}{2}},
\end{align}
where $\im\sqrt{z}\geq \delta$ and
\begin{align}\label{def. sigma(q)}
\sigma(q)=\max\bigl( \, n(\tfrac12-\tfrac1q)-\tfrac12, \, \tfrac{n-1}2(\tfrac12-\tfrac1q)\, \bigr).
\end{align}
Here, $\delta>0$ is an arbitrary but fixed constant.
For $n\geq 3$, $q=2n/(n-2)$, inequality \eqref{Frank and Schimmer} was proved by Dos Santos Ferreira--Kenig--Salo \cite{MR3200351}. 
For the torus this result is due to Shen \cite{MR2366961}. Bourgain--Shao--Sogge--Yao \cite{BourgainShaoSoggeEtAl2015} obtained improved bounds on negatively curved manifolds and on the torus. Frank--Schimmer \cite{MR3620715}, and, independently, Burq--Dos Santos Ferreira--Krupchyk \cite{MR3848231}, proved \eqref{Frank and Schimmer} for $n\geq 2$ at the critical exponent $q_n:=2(n+1)/(n-1)$. 
By elliptic estimates, Sobolev embedding and interpolation with the trivial $L^2\to L^2$ bound, the estimate at the critical exponent implies that \eqref{Frank and Schimmer} holds for all other exponents $2\leq q\leq 2n/(n-2)$. 
We will therefore only consider the critical exponent in the following.  

By duality, the inequality \eqref{Frank and Schimmer} at the critical exponent is equivalent to 
\begin{align}\label{Holder}
\|W_1(-\Delta_g-z)^{-1}W_2\|_{L^{2}(M)\to L^{2}(M)}\leq C_{\delta} |z|^{\frac{1}{q_n}-\frac{1}{2}} \|W_1\|_{L^{n+1}(M)}\|W_2\|_{L^{n+1}(M)},
\end{align}
where we used that $\sigma(q_n)=1/q_n$.
Our main result is the following upgrade of \eqref{Holder} to a stronger bound that replaces the $L^{2}(M)\to L^{2}(M)$ operator norm by a Schatten norm.
\begin{theorem}\label{thm 1}
Let $(M,g)$ be a compact boundaryless Riemannian manifold of dimension $n\geq 2$, with smooth metric $g$.  
Then for every $\delta>0$ there exists a constant $C_{\delta}$ such that
\begin{align}\label{uniform resolvent Schatten}
\|W_1(-\Delta_g-z)^{-1}W_2\|_{\mathfrak{S}^{n+1}(L^2(M))}\leq  C_{\delta} |z|^{\frac{1}{q_n}-\frac{1}{2}}\|W_1\|_{L^{n+1}(M)}\|W_2\|_{L^{n+1}(M)}.
\end{align}
\end{theorem}
Resolvent estimates in Schatten spaces for the Laplace--Beltrami operator on nontrapping asymptotically conic manifolds were established by Guillarmou--Hassell--Krupchyk \cite{MR4150258}. The proof for compact manifolds presented here is much shorter.

Our resolvent estimate is closely related to the following 
spectral cluster bound of Frank--Sabin \cite{MR3682666} (this is the estimate in the form (16) there),
\begin{align}\label{spectral cluster Schatten}
\|W\Pi_{\lambda}\overline{W}\|_{\mathfrak{S}^{n+1}(L^2(M))}\lesssim (1+\lambda)^{\frac{2}{q_n}}\|W\|^2_{L^{n+1}(M)}.
\end{align}
Here $\Pi_{\lambda}:=\mathbf{1}_{[\lambda,\lambda+1]}(\sqrt{-\Delta_g})$ is the spectral projector onto frequencies in the unit length window $[\lambda,\lambda+1]$ for $\lambda>0$ and $A\lesssim B$ means $A\leq CB$ for some unspecified constant $C$. 
We recall that a compact operator $K$ belongs to the Schatten space $\mathfrak{S}^{\alpha}$ if 
\begin{align*}
\|K\|_{\mathfrak{S}^{\alpha}}:=\big(\sum_{j=1}^{\infty}s_j(K)^{\alpha}\big)^{\frac{1}{\alpha}}<\infty,
\end{align*}
where $s_j(K)$ are the singular values of $K$ (i.e.\ the eigenvalues of $(KK^*)^{\frac{1}{2}}$). We refer to, e.g., \cite{MR2154153} for background on Schatten spaces.
The dual estimate to \eqref{spectral cluster Schatten} takes the form
\begin{align}\label{spectral cluster Schatten N functions}
\big\|\sum_{j\in J}\nu_j|f_j|^2\big\|_{L^{q_n/2}(M)}\lesssim (1+\lambda)^{\frac{2}{q_n}}\big(\sum_{j\in J}|\nu_j|^{\frac{n+1}{n}}\big)^{\frac{n}{n+1}},
\end{align}
whenever $(f_j)_{j\in J}\subset \Ran\Pi_{\lambda}$ is an orthonormal family of functions. In the case of a single function, \eqref{spectral cluster Schatten N functions} recovers Sogge's spectral cluster bounds \cite{MR930395} at the critical exponent. One can interpolate \eqref{spectral cluster Schatten} with the sharp counting function remainder estimate of Avakumovi\'c, Levitan and Hörmander \cite{MR80862,MR0058067,MR0609014},
\begin{align*}
\tr W\Pi_{\lambda}\overline{W}=\int_M|W(x)|^2\Pi_{\lambda}(x,x)\rd V_g\lesssim (1+\lambda)^{n-1}\|W\|_{L^2(M)}^2,
\end{align*}
and with the trivial operator norm bound 
\begin{align*}
    \|W\Pi_{\lambda}\overline{W}\|_{L^2(M)\to L^2(M)}\leq \|W\|^2_{L^{\infty}(M)}
\end{align*}
to obtain the full range of spectral cluster estimates
\begin{align}\label{spectral cluster bounds C11}
    \big\|\sum_{j\in J}\nu_j|f_j|^2\big\|_{L^{q/2}(M)}\lesssim(1+\lambda)^{2\sigma(q)}\big(\sum_{j\in J}|\nu_j|^{\alpha(q)}\big)^{1/\alpha(q)},
\end{align}
where 
\begin{align*}
\alpha(q):=\max\bigl( \, \tfrac{q(n-1)}{2n},\tfrac{2q}{q+2}\, \bigr).
\end{align*}
Since 
$\im(-\Delta_g-(\lambda+\I)^2)^{-1}\gtrsim \lambda^{-1}\Pi_{\lambda}$, the resolvent bound \eqref{uniform resolvent Schatten} implies \eqref{spectral cluster Schatten} and thus \eqref{spectral cluster Schatten N functions}, \eqref{spectral cluster bounds C11}. On the other hand, it is not difficult to see that \eqref{spectral cluster Schatten} implies \eqref{uniform resolvent Schatten} with a logarithmic loss, i.e.\ an additional factor of $\log(2+\lambda)$ on the right hand side (the proof is similar to that of \cite[Prop. 3.3]{JCJFA23}). In \cite{JCJFA23}, it was shown that the $L^{q'}(M)\to L^{q}(M)$ resolvent estimates are actually a direct consequence of Sogge's spectral cluster bounds. The idea of the proof is to apply the Christ--Kiselev lemma \cite{MR1809116} to a microlocalized version of the spectral cluster bound. This strategy does not seem to work for Schatten norm bounds, so we give a direct proof of \eqref{uniform resolvent Schatten}. We follow the approach of Dos Santos Ferreira--Kenig--Salo \cite{MR3200351}, which uses the Hadamard parametrix for the resolvent (see also Hörmander \cite[17.4]{MR2304165} and Sogge \cite{MR930395}).

\section{Hadamard parametrix}
Following \cite{MR3200351}, let $T(z)$ be the operator with kernel
\begin{align*}
T(x,y;z)=\chi(x,y)F(x,y;z),
\end{align*}
where $\chi$ is a localization to the diagonal $x=y$, 
\begin{align*}
F(x,y;z)=\sum_{\nu=0}^N a_{\nu}(x,y)F_{\nu}(\rd_g(x,y);z),
\end{align*}
with $N>(n-1)/2$, and 
\begin{align*}
F_{\nu}(|x|;z)=\nu!(2\pi)^{-n}\int_{\R^n}\frac{\e^{\I x\cdot\xi}}{(|\xi|^2-z)^{1+\nu}}\rd\xi.
\end{align*}
The functions $a_{\nu}$ can be recursively chosen such that 
\begin{align}\label{local parametrix equation}
(-\Delta_g-z)T(z)u=\chi(x,x)u+S(z)u,
\end{align}
where the remainder $S(z)=S_1(z)+S_2(z)$ satisfies 
\begin{align}\label{S1 S2}
S_1(x,y;z)=|z|^{\frac{n-1}{4}}\e^{-\sqrt{z}\rd_g(x,y)}b(x,y;z),\quad
S_2(x,y;z)=\mathcal{O}_{\delta}(|z|^{-1/2}),
\end{align}
where $b$ is a smooth function, see (3.9) and the proof of Lemma 4.2 in \cite{MR3200351}. 
As a technical device, we also define, for $0\leq \nu\leq N$ and $1+\nu+\re s\in [0,(n+1)/2]$,
\begin{align*}
F_{\nu+s}(|x|;z)=\Gamma(1+\nu+s)(2\pi)^{-n}\int_{\R^n}\frac{\e^{\I x\cdot\xi}}{(|\xi|^2-z)^{1+\nu+s}}\rd\xi.
\end{align*}
These kernels define an analytic family of operators $s\mapsto F_{\nu+s}(z)$ which coincide with $F_{\nu}(z)$ for $s=0$. 

\begin{lemma}\label{lemma Fnu+s}
For $|z|=1$, $z\neq 1$, we have
\begin{align*}
\|F_{\frac{n-1}{2}+\I t}(z)\|_{L^1(\R^n)\to L^{\infty}(\R^n)}+\|F_{-1+\I t}(z)\|_{L^2(\R^n)\to L^{2}(\R^n)}\leq C\e^{c|t|^2}.
\end{align*}
\end{lemma}

\begin{proof}
The bound for the first term can be found in \cite[(50)]{MR3730931} and is essentially contained in~\cite{MR894584}. The bound for the second term is trivial.
\end{proof}

\section{Proof of Theorem \ref{thm 1}}
\begin{proof}
We first consider the case $W_2=\overline{W_1}$ and then prove the general case at the end. In the following, we suppress the dependence of constants on $\delta$. After rescaling, Lemma \ref{lemma Fnu+s} together with \cite[Proposition 1]{MR3730931} yields the parametrix bound
\begin{align}\label{THad}
\|W_1T(z)W_2\|_{\mathfrak{S}^{n+1}}\lesssim |z|^{\frac{1}{q_n}-\frac{1}{2}}\|W_1\|_{L^{n+1}}\|W_2\|_{L^{n+1}}.
\end{align}
After summing \eqref{local parametrix equation} over a partition of unity, we obtain 
\begin{align}\label{global parametrix equation}
(-\Delta_g-z)T(z)=\mathbf{1}+S(z),
\end{align}
where $S(z)$ satisfies
\begin{align}\label{S(z)W}
\|S(z)W\|_{\mathfrak{S}^{2(n+1)}}\lesssim |z|^{\frac{1}{2q_n}}\|W\|_{L^{n+1}}.
\end{align}
This follows from \eqref{S1 S2} and \cite[Corollary 3]{MR3682666} for $S_1(z)$, whereas for $S_2(z)$, \eqref{S1 S2} yields the stronger bound
\begin{align*}
\|S_2(z)W\|_{\mathfrak{S}^{2(n+1)}}\leq \|S_2(z)W\|_{\mathfrak{S}^{2}}\lesssim |z|^{-\frac{1}{2}}\|W\|_{L^2}\lesssim |z|^{-\frac{1}{2}}\|W\|_{L^{n+1}}. 
\end{align*}
In the last estimate, we used H\"older's inequality and the fact that $M$ is compact.

Note that, even though $T(z)$ and $S(z)$ are not symmetric, their adjoints satisfy the same bounds \eqref{THad}, \eqref{S(z)W}. Moreover, these estimates remain true when replacing $z$ by $\overline{z}$. 
Applying the resolvent operator $(-\Delta_g-z)^{-1}$ to \eqref{global parametrix equation} yields
\begin{align}\label{identity}
W_1(-\Delta_g-z)^{-1}W_2= W_1T(z)W_2+W_1(-\Delta_g-z)^{-1}S(z)W_2
\end{align}
and similarly for $z$ replaced by $\overline{z}$.
We use the noncommutative Hölder inequality \cite[Theorem 2.8]{MR2154153} and~\eqref{S(z)W} to bound
\begin{align}\notag
\|W_1(-\Delta_g-z)^{-1}S(z)W_2\|_{\mathfrak{S}^{n+1}}&\leq \|W_1(-\Delta_g-z)^{-1}\|_{\mathfrak{S}^{2(n+1)}}\|S(z)W_2\|_{\mathfrak{S}^{2(n+1)}}\\
&\lesssim |z|^{\frac{1}{2q_n}}\|W_1(-\Delta_g-z)^{-1}\|_{\mathfrak{S}^{2(n+1)}}\|W_2\|_{L^{n+1}}.\label{estimate 1}
\end{align}
Then we use the resolvent identity
\begin{align*}
(-\Delta_g-z)^{-1}(-\Delta_g-\overline{z})^{-1}=\frac{1}{2\im z}((-\Delta_g-z)^{-1}-(-\Delta_g-\overline{z})^{-1})
\end{align*}
and the fact that $|\im z|\gtrsim |z|^{\frac{1}{2}}$ in the considered range to obtain
\begin{align}\notag
&\|W_1(-\Delta_g-z)^{-1}\|_{\mathfrak{S}^{2(n+1)}}=\|W_1(-\Delta_g-z)^{-1}(-\Delta_g-\overline{z})^{-1}\overline{W_1}\|_{\mathfrak{S}^{n+1}}^{1/2}\\
&\lesssim|z|^{-\frac{1}{4}}(\|W_1(-\Delta_g-z)^{-1}\overline{W_1}\|_{\mathfrak{S}^{n+1}}+\|W_1(-\Delta_g-\overline{z})^{-1}\overline{W_1}\|_{\mathfrak{S}^{n+1}})^{1/2}.\label{estimate 2}
\end{align}
Combining \eqref{estimate 1}--\eqref{estimate 2} and using that $ab\lesssim \epsilon a^2+\epsilon^{-1}b^2$ for arbitrary $\epsilon>0$, we get
\begin{align}\notag
\|W_1(-\Delta_g-z)^{-1}S(z)W_2\|_{\mathfrak{S}^{n+1}}&\lesssim \epsilon^{-1} |z|^{-\frac{1}{2}}(\|W_1(-\Delta_g-z)^{-1}\overline{W_1}\|_{\mathfrak{S}^{n+1}}\\
&+\|W_1(-\Delta_g-\overline{z})^{-1}\overline{W_1}\|_{\mathfrak{S}^{n+1}})
+\epsilon|z|^{\frac{1}{q_n}}\|W_2\|^2_{L^{n+1}}\label{WresolventS(z)}
\end{align}
and similarly for $z$ replaced by $\overline{z}$.
We choose $\epsilon=C|z|^{-\frac{1}{2}}$ for some sufficiently large constant $C$. Then, using \eqref{THad}, \eqref{S(z)W}, \eqref{identity}, \eqref{WresolventS(z)} and their analogues for $\overline{z}$, and recalling that $W_2=\overline{W_1}$, we can absorb the $\epsilon^{-1}|z|^{-\frac{1}{2}}$ term to get \eqref{uniform resolvent Schatten} in this case.

Turning to the general case, in view of \eqref{THad}, it remains to prove 
\begin{align*}
    \|W_1(-\Delta_g-z)^{-1}S(z)W_2\|_{\mathfrak{S}^{n+1}}\lesssim  |z|^{\frac{1}{q_n}-\frac{1}{2}}\|W_1\|_{L^{n+1}}\|W_2\|_{L^{n+1}}.
\end{align*}
By the same arguments as before, this follows from \eqref{S(z)W} and the bound
\eqref{uniform resolvent Schatten} with $W_2=\overline{W_1}$, which we have already proved.
\end{proof}

\noindent\textbf{Acknowledgements:} The author thanks Rupert Frank, Mikko Salo, Jeff Galkowksi and Xiaoyan Su for useful discussions and correspondence.
Support by the Engineering \& Physical Sciences Research Council [grant number EP/X011488/1] is acknowledged.

\bibliographystyle{abbrv}

\end{document}